\documentclass[a4paper]{article}
\usepackage[margin=25mm]{geometry}
\usepackage{amsmath}
\usepackage{amsthm}
\usepackage{amsfonts}
\usepackage{amssymb}
\usepackage{graphicx}
\usepackage{verbatim}
\usepackage{comment}
\usepackage{enumerate}
\usepackage{float}
\immediate\write18{texcount -tex -sum  \jobname.tex > \jobname.wordcount.tex}

\newtheorem{theorem}{Theorem}

\newtheorem{lemma}{Lemma}[section]
\theoremstyle{remark}
\newtheorem*{remark}{Remark}

\providecommand{\keywords}[1]
{
  \small	
  \textbf{Keywords:} #1
}

\providecommand{\subjclass}[2][2020]
{
  \small	
  \textbf{Mathematics Subject Classification(2020):} #2
}

\title{Non-Isomorphic Groups with Isomorphic Power and Commuting Graphs}
\author{Surbhi$^1$ and Geetha Venkataraman$^2$ \\
        \small $^{1,2}$School of Liberal Studies, Dr. B. R. Ambedkar University Delhi, India. \\
}
\date{} % Comment this line to show today's date

\begin{document}
\maketitle

\begin{abstract}
The power graph of a group $G$ is a graph with vertex set $G$, in which two vertices are adjacent if one is some power of the other. In the commuting graph, with $G$ as the vertex set, two vertices are joined by an edge if they commute in $G$. The enhanced power graph of a group $G$ is a graph with vertex set $G$ and an edge joining two vertices $x$ and $y$ if $\langle x,y \rangle$ is cyclic. In this paper, we answer a question posed by P. J. Cameron, namely, if there exist groups $G$ and $H$ such that the power graph of $G$ is isomorphic to the commuting graph of $H$. We show that the answer is yes if $G$ is the generalised quaternion group and $H$ is the dihedral group. We also show that the enhanced power graph of the dicyclic group is isomorphic to the commuting graph of the dihedral group. 
\end{abstract} \hspace{10pt}

%TC:ignore



\section{Introduction}

In this paper, we consider undirected graphs whose vertex set is a finite group and whose edges are joined based on some group-theoretic property of the vertex set. Throughout this paper, the graphs considered are not allowed multiple edges or loops. The power graph Pow($G$) of a group $G$ is a graph in which two vertices $x$ and $y$ share an edge if either $x=y^m$ or $y=x^m$ for some integer $m$. The commuting graph Com($G$) of $G$ is a graph in which two vertices $x$ and $y$ are adjacent if $x$ and $y$ commute in $G$. In 2017, while trying to characterise when the power graph and commuting graph of a finite group are equal, Aalipour et al. \cite{Aalipour} defined a new graph, called the enhanced power graph, whose vertex set consists of all elements of a group $G$, in which two vertices $x$ and $y$ are adjacent if they generate a cyclic group. We denote this graph as EPow($G$).\\

Let $E$($\Gamma$) denote the edge set of the graph $\Gamma$. It is easy to see that $E$(Pow($G$)) $\subseteq$ $E$(EPow($G$)) $\subseteq$ $E$(Com($G$)). In \cite{Aalipour}, a necessary and sufficient condition was given for a finite group $G$ to satisfy Pow($G$) = Com($G$). Then, in \cite{PJC_review} P. J. Cameron asked the following question \textit{Do there exist groups $G$ and $H$ such that} Pow($G$)\textit{ is isomorphic to} Com($H$)\textit{?} He states that this will be true if Pow($G$) = Com($G$) and $G$ and $H$ have isomorphic commuting graphs, or if Pow($H$) = Com($H$) and $G$ and $H$ have isomorphic power graphs.
%rephhrase now that the sec are charned

In this paper, we explore this question by considering the generalised quaternion group and the dihedral group whose order is some power of 2 and we show that even though they are not isomorphic as groups, the following holds.

\begin{theorem}\label{theo1}
    Let $Q_{2^{n+1}}$ denote the generalised quaternion group of order $2^{n+1}$ and $D_{2.2^n}$ denote the dihedral group of the same order. For each $n \in \mathbb{N}$, we have {\rm Pow(}$Q_{2^{n+1}}${\rm )} $\cong$ {\rm Com(}$D_{2.2^n}${\rm )}.
\end{theorem}

The order of the generalised quaternion group is some positive power of 2 but the dihedral group can have an order divisible by an odd prime too. The generalised quaternion group is a special case of the dicyclic group which has order $4m$ where $m$ is a positive integer. This motivated us to check whether Theorem \ref{theo1} holds for the power graph of $Q_{4m}$ and the commuting graph of $D_{2.2m}$. We show that although Pow($Q_{4m}$) is not necessarily isomorphic to Com($D_{2.2m}$), the following result is analogous to Theorem \ref{theo1}.

\begin{theorem}\label{theo2}
    Let $m \in \mathbb{N}$. If there exists an odd prime $p$ that divides $m$ then  {\rm EPow(}$Q_{4m}${\rm )} $\cong$ {\rm Com(}$D_{2.2m}${\rm )} where $Q_{4m}$ is the dicyclic group and $D_{2.2m}$ is the dihedral group.
\end{theorem}
Note that if no odd prime $p$ divides $m$, then $Q_{4m} = Q_{2^{n+1}}$ for a suitable $n$ and Theorem \ref{theo1} holds.\\
In Section 2, we consider the presentations of the generalised quaternion group, the dicyclic group and the dihedral group. We provide some of their properties which will help us to find the power graph of the generalised quaternion group, the enhanced power graph of the dicyclic group and the commuting graph of dihedral group. In Section 3, we prove Theorem \ref{theo1} and Theorem \ref{theo2}. While proving these theorems, we also give the structure of the graphs involved. The structure of power graphs of certain finite groups are also discussed in \cite{mckemmie}. We will denote the center of a group $G$ by $Z(G)$ and the centraliser of an element $x$ by $C_G(x)$.
%Finally, in Section 4, we examine other graphs of these groups.
 
%==================================================================================================================================================================================

\section{Generalised quaternion group, dihedral group and dicyclic group}

%------------------------- PROPERTIES OF GEN. QUATERNION GROUP ------------------------------------

The quaternion group $Q_8$ which is defined as $\langle h,x \mid h^4 =e , x^2=h^2, x^{-1}hx=h^{-1} \rangle $, has a more general form where the order of the group is a power of 2. We call it the generalised quaternion group and its presentation is $$Q_{2^{n+1}} = \langle h,x \mid h^{2^n} =e, x^2 = h^{2^{n-1}}, x^{-1}hx=h^{-1} \rangle .$$ 
%presentation from Robinson's book
The following lemma describes some well-known properties of $Q_{2^{n+1}}$ but for the sake of completeness, we provide the proof.

%-------Lemma-------
\begin{lemma}\label{prop_quat} Let $Q_{2^{n+1}}$ be the generalised quaternion group of order $2^{n+1}$.
\begin{enumerate}[{\rm (i)}]
    \item Each element of $Q_{2^{n+1}}$ is of the form $h^ix^j$ where $0 \leq i < 2^n$ and $j \in \{0,1\}$. \label{prop_quat_form}
    \item $|h^ix| = 4$ and its inverse is given by $h^jx$ where $ j=i+2^{n-1} $ for every $0 \leq i < 2^n$. \label{prop_quat_inverse}
    \item $Q_{2^{n+1}}$ has a unique element of order 2, namely $x^2 = h^{2^{n-1}}$. \label{prop_quat_convolution}
    %\item $Q_{2^{n+1}}$ has $2^{n-1}$ cyclic subgroups of the form $ \langle h^ix \rangle $ of order $4$.
    \item $Z(Q_{2^{n+1}}) = \{ e, x^2 \}$.
\end{enumerate}
\end{lemma}
\begin{proof}
    The group $Q_{2^{n+1}}$ is generated by two elements $h$ and $x$. So, any element $a$ of $Q_{2^{n+1}}$ will be of the form $y_1^{\alpha_1}y_2^{\alpha_2}...y_k^{\alpha_k}$ where $y_m$ is either $h$ or $x$ and $k$ is some positive integer. We will show that $a=h^ix^j$ for some $0 \leq i < 2^n$ and $j \in \{ 0,1 \}$. By manipulating the relation $x^{-1}hx=h^{-1}$, we get that $xh=h^{2^{n-2} + 2^{n-1}}x$ and $a$ can be written as $a=h^ix^j$ for some $0 \leq i < 2^n$ and $0 \leq j < 4$. Also, since $x^2 = h^{2^{n-1}}$, we have $hx^2 = h^{2^n} = h^0$ and $hx^3=h^0x$. Therefore, $a=h^ix^j$ for some $0 \leq i < 2^n$ and $j \in \{0,1 \}$ which proves the first part. For the second part, observe that using induction we get $a^2=x^2$ for every $a=h^ix$. Since $x$ has order 4, $a$ has order 4. For the inverse of $a$, let $b=h^jx$ be the inverse of $a$ for some $0\leq j < 2^n$. Using the relation $hx=xh^{-1}$, we get that $h^ixh^jx=x^2h^{-(j-i)}$ which implies that $x^2h^{-(j-i)} = e$ if and only if $j=i+2^{n-1}$. To see that $Q_{2^{n+1}}$ has a unique element of order 2, consider $H=\langle h \rangle $ which is a subgroup of index 2. The subgroup $H$ has a unique element of order 2 since it is cyclic and each element of the coset $Hx$ has order 4 as shown above. The center contains $x^2$ as it is contained in every cyclic subgroup of $Q_{2^{n+1}}$ which implies that it is contained in the centraliser of each element. To see that no other non-trivial element belongs to the center, consider distinct elements $a=h^ix$, $b=h^jx$ and $c=h^k$ for some $0 \leq i,j < 2^n$ and $ 1 \leq k < 2^n$. Observe that $a$ commutes with $b$ if and only if $a=b^{-1}$ and $a$ commutes with $c$ if and only if $k=2^{n-1}$.
    %\item The cyclic subgroup generated by $h^ix$ has 4 elements: $h^ix, x^2 = h^{2^{n-1}},h^jx,e$ where $j=i+2^{n-1}$. Since i takes total $2^n$ values, therefore there are a total of $2^n/2 = 2^{n-1}$ cyclic subgroups of the form $\langle h^ix \rangle $.
\end{proof}

%-------------PROPERTIES OF DICYCLIC GROUP--------------------------------------------------------

The generalised quaternion group is a special case of the dicyclic group which is defined as: 
            $$Q_{4m} = \langle h,x \mid h^{2m} = e, h^m=x^2, x^{-1}hx=h^{-1} \rangle .$$
If we substitute $m=2^{n-1}$, we get the generalised quaternion group of order $2^{n+1}$. 

%---Lemma------
\begin{lemma}\label{prop_dicyclic} Let $Q_{4m}$ be the dicyclic group of order $4m$ where $m$ is a positive integer. 
\begin{enumerate}[{\rm (i)}]
    \item If $y \in Q_{4m}$ then $y=h^ix^j$ for some $0 \leq i < 2m$ and $j \in \{0,1\}$. \label{prop_dicyclic_element}
    \item If $y=h^ix$ for some $0 \leq i < 2m$ then $|y| = 4$ and $y^{-1} = h^{i+m}j$.
    \item There exists a unique element of order 2 in $Q_{4m}$, namely $x^2 = h^{m}$.
    \item $Z(Q_{4m}) = \{ e, x^2 \}$.
    %\item $Z(Q_{4m})$ = $\{ e, h^m \}$
    %\item $C_{Q_{4m}}(h^i) = \langle h \rangle \  \forall \  0 \leq i < 2m $
    %\item $C_{Q_{4m}}(h^ix) = \{ e,h^ix,h^m, h^{i+m}x \}  \  \forall \  0 \leq i < 2m$ 
\end{enumerate}
\end{lemma}

\noindent It is easy to verify that the above lemma holds by following a proof similar to Lemma \ref{prop_quat}.

%---------------PROPERTIES OF DIHEDRAL GROUP-------------------------------------------------------

The dihedral group is the set of symmetries of a regular $n$-gon where $n$ can be any positive integer. It has $2n$ elements. However, for the scope of this paper we need the order to be a multiple of 4 so we consider $n$ to be even. Let $D_{2.2m}$ denote the dihedral group of order $4m$ for any positive integer $m$. Then its presentation is
     $$D_{2.2m} = \langle a,b \mid a^{2m}=b^2=e, bab = a^{-1} \rangle.$$
%--Lemma-----
\begin{lemma}\label{prop_dihedral_4m} Let $D_{2.2m}$ be the dihedral group of order $4m$.
    \begin{enumerate}[{\rm (i)}]
        \item If $r$ is a rotation in $D_{2.2m}$ then $r=a^i$ for some $0 \leq i < 2m$.
        \item If $r$ is a reflection in $D_{2.2m}$ then $r=a^ib$ for some $0 \leq i < 2m$.
        \item $Z(D_{2.2m})$ = $\{ e, a^m \}$. \label{prop_dihedral_center}
        \item $C_{D_{2.2m}}(a^i) = \langle a \rangle$ for every $0 \leq i < 2m $ and $i \neq m$.\label{prop_dihedral_cyclic_subgroup}
        \item $C_{D_{2.2m}}(a^ib) = \{ e,a^ib,a^m, a^{i+m}b \} $ for every $0 \leq i < 2m$. \label{prop_dihedral_centraliser}
    \end{enumerate}
\end{lemma}
\begin{proof}
    The first three parts are very well-known. Consider three distinct elements of $D_{2.2m}$ denoted by $s=a^ib$, $t=a^jb$ and $u=a^k$ for some $0\leq i,j < 2m$ and $1 \leq k < 2m$. Then, using the relation $ab = ba^{-1}$ we can see that $s$ commutes with $t$ if and only if $a^{2(j-i)}=e$ which gives us $j=i+m$. Also, using the same relation we get that $s$ and $u$ commute if and only if $a^{2k}=e$, that is, $k=m$. This proves the remaining parts.
\end{proof}

%--------- REMARK -------------------

Note that in the presentation of the dihedral group stated above, by substituting $m=2^{n-1}$, we can get the order to be equal to that of the generalised quaternion group. This will help us to compare their graphs in the next section. By the same substitution in the Lemma \ref{prop_dihedral_4m}, we get the properties of dihedral group of order $2^{n+1}$.

%===================================================================================================================================================================================

\section{The Isomorphisms}

%add because of pjc result, all three graphs of qaut are equal whereas dih 2n+! , the epw and com are not equal bec.. ; but the pow graph will be equal to epow graph.
%\begin{figure}
%\centering
%\includegraphics[scale=0.45]{Pow_Q_32_Labelled.png}
%\caption{Power graph of generalised quaternion group of order 32}
%\end{figure}

The generalised quaternion group $Q_{2^{n+1}}$ is not isomorphic to the dihedral group $D_{2.2^n}$ as $Q_{2^{n+1}}$ has a unique involution whereas $D_{2.2^n}$ does not, but we show an isomorphism exists between their power graph and commuting graph respectively. We use the presentations and notations established in the previous section. We use the symbol $ \Gamma_1 \cup \Gamma_2$ to describe the disjoint union of graphs, the symbol $ \Gamma_1 \underline\cup \Gamma_2$ to describe the non-disjoint union of graphs and the symbol $ \Gamma_1 \nabla \Gamma_2$ to describe the graph join of graphs $\Gamma_1$ and $\Gamma_2$. The union of graphs $ \Gamma_1$ and $ \Gamma_2$ is a graph whose vertex set is the union of $V(\Gamma_1)$ and $V(\Gamma_2)$ and the edge set is the union of $E(\Gamma_1)$ and $E(\Gamma_2)$. In $ \Gamma_1 \cup \Gamma_2$, we consider the vertices and edges of $ \Gamma_1$ and $\Gamma_2$ distinct regardless of their labels whereas in $ \Gamma_1 \underline\cup \Gamma_2$, we take the graph union by merging labelled vertices and edges. We consider $ \Gamma_1$ and $\Gamma_2$ with disjoint vertices to define the graph join. The graph join $ \Gamma_1 \nabla \Gamma_2$ is a graph whose vertex set is $V(\Gamma_1) \cup V(\Gamma_2)$ and the edge set is $E(\Gamma_1) \cup E(\Gamma_2)$ together with all the edges joining $V(\Gamma_1)$ and $V(\Gamma_2)$. The symbol $K_n$ is used for a complete graph on $n$ vertices.

%----------------------- THEOREM 1----------------------------------------------------

%\begin{figure*}[!b]
%\centering
%\includegraphics[scale=0.45]{Pow_Q_32_Labelled.png}
%\caption{Power graph of the generalised quaternion group of order 32}
%\end{figure*}

\setcounter{theorem}{0}
\begin{theorem}\label{theorem1}
     Let $Q_{2^{n+1}}$ denote the generalised quaternion group of order $2^{n+1}$ and $D_{2.2^n}$ denote the dihedral group of the same order. For each $n \in \mathbb{N}$, {\rm Pow(}$Q_{2^{n+1}}${\rm )} $\cong$ {\rm Com(}$D_{2.2^n}${\rm )}.
\end{theorem}
\begin{proof}

%  1. Show when two vertices are adjacent in $Q_{2^{n+1}}$. 
Our first step will be to figure out the edge set of Pow($Q_{2^{n+1}}$). Note that in Pow($G$), the identity element of $G$ is adjacent to every vertex. As seen in Lemma \ref{prop_quat}(\ref{prop_quat_form}), the elements of $Q_{2^{n+1}}$ are either of the form $h^i$ or $h^ix$ for $0 \leq i < 2^n$.  Theorem 2.12 of \cite{Chakrabarty_Ghosh} states that for a finite cyclic $p$-group $G$, the power graph of $G$ is complete. Using this, we can see that if $s=h^i$ then $s$ is adjacent to $h^j$ for every $0 \leq i,j < 2^n$. In addition, the unique convolution $h^{2^{n-1}}$ is adjacent to every vertex of the form $s=h^ix$ as $s^2=x^2=h^{2^{n-1}}$. Now, let $t=h^ix$. Clearly, $t$ is adjacent to $t^{-1}$. For any other $s \not\in \langle t \rangle$, we see that $\langle s \rangle \cap \langle t\rangle = \{ e, x^2\}$. Therefore, $s$ and $t$ are not adjacent. So, we have Pow($Q_{2^{n+1}}$) = $(K_{2^n-2}$ $\cup$ $(2^{n-1})K_2)$ $\nabla$ $K_2$.

%2. Define a map $f:Q_{2^{n+1}} \rightarrow D_{2.2^n}$ such that $h^ix^j \mapsto a^ib^j$ for every $0 \leq i.j < 2^n$
Consider a function $f:Q_{2^{n+1}} \rightarrow D_{2.2^n}$ that maps $h^ix^j$ to $a^ib^j$ for $0 \leq i,j < 2^n$. This map is a graph isomorphism. Clearly, $f$ is one-one and onto. 
%3. Show their images are also adjacent in $D_{2.2^n}$.
For $f$ to be edge-preserving, we show that $s,t \in Q_{2^{n+1}}$ are adjacent in the power graph if and only if $f(s),f(t)$ are adjacent in the commuting graph of $D_{2.2^n}$. In Com($G$), the identity of group $G$ is universal vertex as it commutes with every element of $G$. The function $f$ maps identity of $Q_{2^{n+1}} $ to identity of $D_{2.2^n}$ and both are universal vertices in their respective graphs. The function $f$ maps $h^{2^{n-1}}$ to $a^{2^{n-1}}$. By Lemma \ref{prop_dihedral_4m}(\ref{prop_dihedral_center}), we have $a^{2^{n-1}}$ is central and hence it is also a universal vertex in the commuting graph of $D_{2.2^n}$. Now, $h^i$ is mapped to $a^i$ under $f$. By Lemma \ref{prop_dihedral_4m}(\ref{prop_dihedral_cyclic_subgroup}), the centraliser of $a^i$ is $\langle a \rangle$, therefore $a^i$ is adjacent to $a^j$ for every $0 \leq j <2^n$. Lastly, $f(h^ix)=a^ib$ and by Lemma \ref{prop_dihedral_4m}(\ref{prop_dihedral_centraliser}), we have $a^ib$ is adjacent to $a^{i+2^{n-1}}b$, that is, $f(h^{i+2^{n-1}}x)$ is adjacent to $f(h^ix)$. 
%4. Show that if two vertices are not adjacent in $Q_{2^{n+1}}$ then they are not adjacent in $D_{2.2^n}$.
Conversely, we show that if two vertices are not adjacent in Pow($Q_{2^{n+1}}$) then their images are also not adjacent in Com($D_{2.2^n}$). Let $s,t \in Q_{2^{n+1}} \symbol{92} \{e, h^{2^{n-1}} \}$ such that $s=h^i$ and $t=h^jx$ for some $i,j$. As shown above, an edge does not join the vertices $s$ and $t$. By Lemma \ref{prop_dihedral_4m}, $f(s)=a^i$ and $f(t)=a^jb$ commute if and only if $i=2^{n-1}$, that is, $s=h^{2^{n-1}}$ which is a contradiction. Hence, $f(s)$ and $f(t)$ are also not adjacent. Also, if $s,t \in Q_{2^{n+1}} \symbol{92} \{e, h^{2^{n-1}} \}$ such that $s=h^ix$ and $t=h^jx$ for some $j \neq i+2^{n-1}$. Then, they are not adjacent in Pow($Q_{2^{n+1}}$). Their images $f(s)=a^ib$ and $f(t)=a^jb$ are also not adjacent whenever $j \neq i+2^{n-1}$. This proves that $f$ is edge-preserving.
\end{proof} 

The figure below illustrates the power graph of the generalised quaternion group of order 32. As we can see that Pow($Q_{32}$) = $(K_{14}$ $\cup$ $8K_2)$ $\nabla$ $K_2$. The set $\langle h \rangle \symbol{92} Z(Q_{32})$ is the vertex set of $K_{14}$ and there are 8 complete graphs on the vertex sets of the form $\{ h^ix, h^{i+8}x \}$ where $0 \leq i \leq 7$, that is, 8 copies of $K_2$. We take the union of these graphs and join it with a complete graph on $\{ e, h^8 \}$.

\begin{figure}[H]
\centering
\includegraphics[scale=0.45]{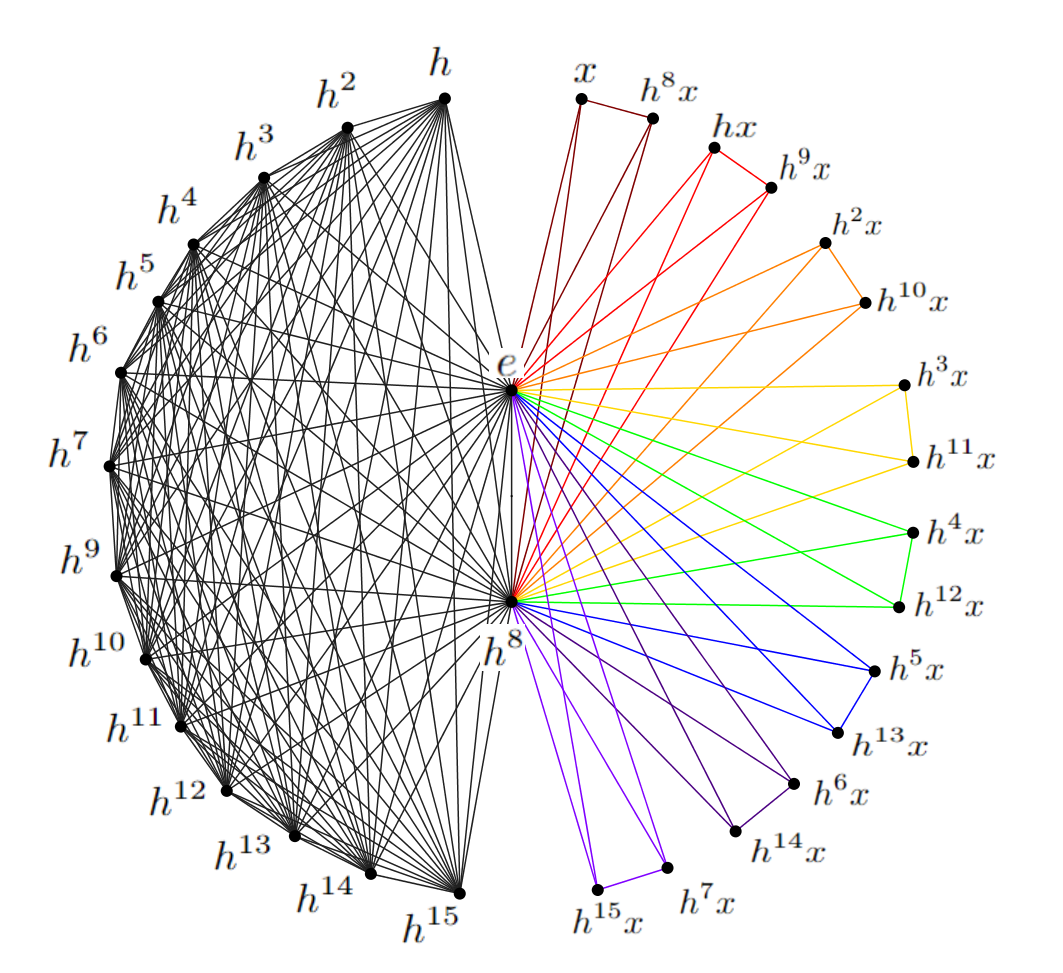}
\caption{Power graph of the generalised quaternion group of order 32}\label{quat_pow}
\end{figure}

\begin{remark}
    For a finite group $G$, we have $E$(Pow($G$)) $\subseteq$ $E$(EPow($G$)) $\subseteq$ $E$(Com($G$)). Proposition 3.2 of \cite{PJC_review} states that Pow($G$) = EPow($G$) if and only if $G$ does not have a subgroup isomorphic to $C_p \times C_q$ for distinct primes $p,q$ and EPow($G$) = Com($G$) if and only if $G$ does not have a subgroup isomorphic to $C_p \times C_p$ for prime $p$. Using this result we get that for the generalised quaternion group, all three graphs are equal. Since the dihedral group of order $2^{n+1}$ has a subgroup isomorphic to $C_2 \times C_2$, we have Pow($D_{2.2^n}$) = EPow($D_{2.2^n}$) $\neq$ Com($D_{2.2^n}$).
\end{remark}

%----------------------------------------THEOREM 2 ------------------------------------------------------------
The above theorem motivated us to check the isomorphism between the dicyclic group's power graph and the dihedral group's commuting graph. In the next result, we describe the structure of the power graph of dicyclic group and the commuting graph of dihedral group and we show that the two can't be isomorphic. Instead, there exists an isomorphism between the enhanced power graph of the dicyclic group and the commuting graph of the dihedral group. 

\begin{theorem}\label{theorem2}
    Let $Q_{4m}$ be the dicyclic group and $D_{2.2m}$ be the dihedral group of order $4m$ for integer $m \geq 2$. If there exists an odd prime $p$ that divides $m$ then {\rm EPow(}$Q_{4m}${\rm )} $\cong$ {\rm Com(}$D_{2.2m}${\rm )}.
\end{theorem}
\begin{proof}
    %Add lemma to describe Pow($Q_{4m}$).
    Consider the Pow($Q_{4m}$). The identity element is a universal vertex in Pow($Q_{4m}$). By Lemma \ref{prop_dicyclic}(\ref{prop_dicyclic_element}), an element of $Q_{4m}$ is either of the form $h^i$ or $h^ix$. Let $s \in Q_{4m}$. If $s=h^ix$, then $s^2=x^2=h^{m}$. Therefore, $s$ is adjacent to $h^m$. Let $t=h^jx$ for $i \neq j$ . By Lemma \ref{prop_dicyclic}, the elements $s,t$ are adjacent if and only if $s$ is the inverse of $t$, that is, $j=i+m$. 
    If an odd prime $p$ divides $m$, then $Q_{4m}$ has an element of order $p$, say $v$. As opposed to what we saw in Pow($Q_{2^{n+1}}$), the induced subgraph on $\langle h \rangle$ can not be complete because we will not have an edge joining the element $h^m$ and the element $v$ as their orders are co-prime. Lastly, the elements $s=h^i$ and $t=h^jx$ are not adjacent in Pow($Q_{4m}$) for $s,t \in Q_{4m} \symbol{92} \{ e,h^m \}$. So,
    Pow($Q_{4m}$) = Pow($\langle h \rangle$) $\underline\cup$ (m$K_2$ $\nabla$ Pow($Z(Q_{4m})$)).
    
    %Add lemma to describe Com($D_{2.2m}$).
    In the commuting graph of $D_{2.2m}$, elements $e$ and $a^m$ are central by Lemma \ref{prop_dihedral_4m} and therefore adjacent to every other vertex. The induced subgraph on $\langle a \rangle$ is complete. Let $t_1=a^ib$ and $t_2=a^jb$ be distinct. Then, they are adjacent if and only if $j=i+m$. For $i \neq m$, no element of the form $a^i$ is adjacent to any element of the form $a^jb$. Therefore, Com($D_{2.2m}$) = $( K_{2m-2} \cup mK_2 ) \nabla K_2$.
    
    %Compare the number of edges in Pow($Q_{4m}$) and Com($D_{2.2m}$) which gives that the two graphs can’t be isomorphic.
    To show that both of these graphs are not isomorphic, we count the number of edges in each of them. In Com($D_{2.2m}$), we have a complete graph on $2m$ vertices and $m$ copies of a complete graph on 4 vertices but all these complete graphs have an edge in common which is the edge that joins $e$ and $a^m$. We know that the number of edges in a complete graph on $n$ vertices is $n(n-1)/2$. So, we get $m(2m-1) + 5m$ edges in Com($D_{2.2m}$). However, in Pow($Q_{4m}$), we have $m$ copies of a complete graph on 4 vertices but the induced subgraph on $\langle h \rangle$ is not complete. Therefore, we have strictly less than $m(2m-1) + 5m$ edges in Pow($Q_{4m}$). Hence, these two graphs can not be isomorphic to each other.
    Now, consider EPow($Q_{4m}$). The induced subgraph on $\langle h \rangle$ is complete and all other edges in EPow($Q_{4m}$) lie in Pow($Q_{4m}$) which gives us EPow($Q_{4m}$) = $(K_{2m-2}$ $\cup$ $m K_2)$ $\nabla$ $K_2$.
    So, we define a map $f$ from $Q_{4m}$ to $D_{2.2m}$ which sends $h^ix^j$ to $a^ib^j$. This map is one-one, onto as well as edge-preserving and hence, we have our result.    
\end{proof}

The figure below describes the power graph of the dicyclic group of order $4m$ which aligns with the structure mentioned in the above proof.

\begin{figure}[H]
\centering
\includegraphics[scale=0.45]{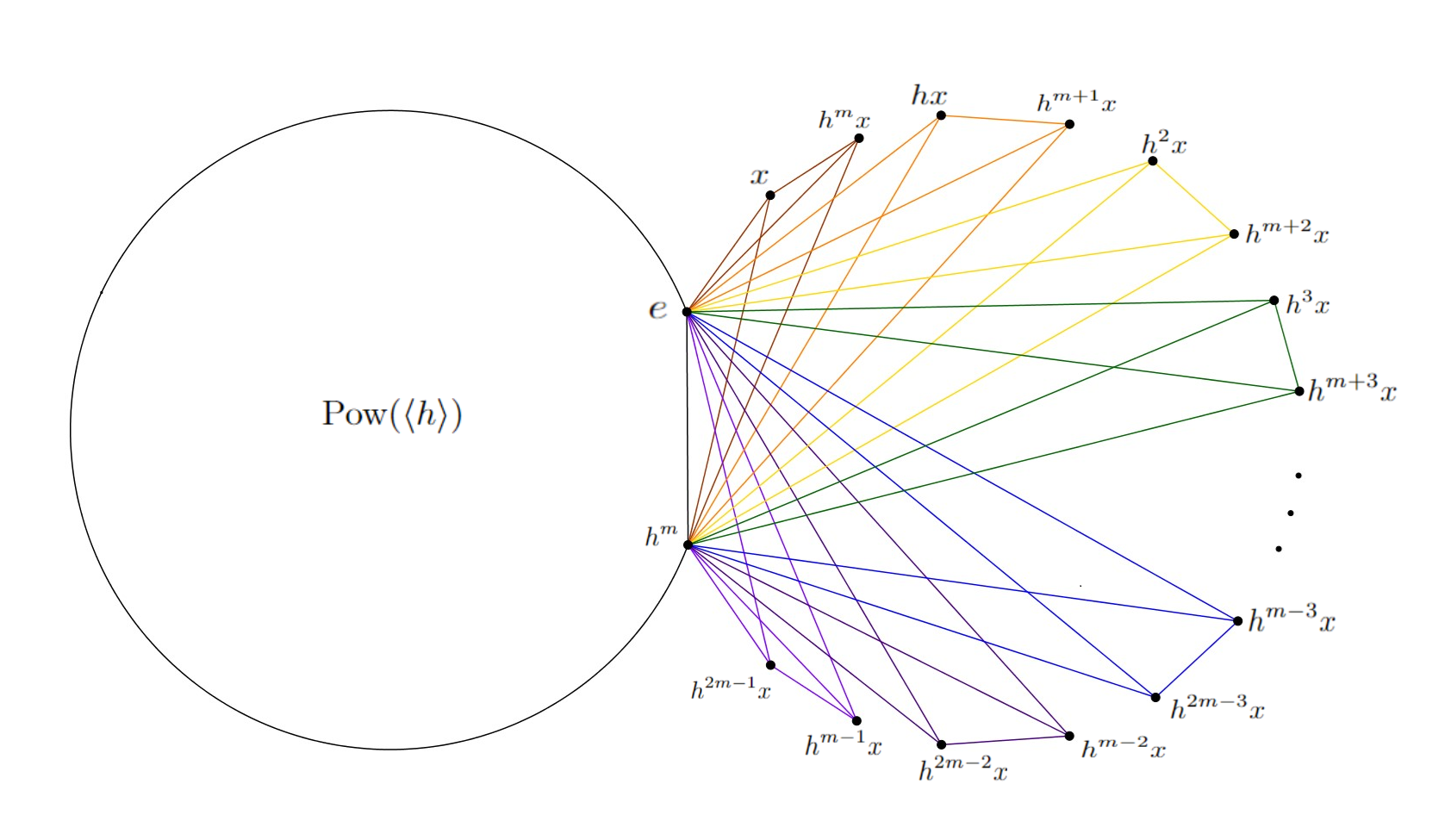}
\caption{Power graph of the dicyclic group of order $4m$}\label{dicyclic_pow}
\end{figure}
 
\begin{remark}
    By the result stated in the previous remark, we can see that Pow($Q_{4m}$) $\neq$ EPow($Q_{4m}$) = Com($Q_{4m}$) and Pow($D_{2.2m}$) $\neq$ EPow($D_{2.2m}$) $\neq$ Com($D_{2.2m}$).
\end{remark} 

%===================================================================================================================================================================================

%\section{Miscellaneous}

%add another misc section
%add pjc equality theorem
%we examine other graphs of the groups
%in gen epow pow and com of dihedral - are they equal
%in quat also, are they equal
%
%add graphs. pehle general then add specific q32.
%------------------------------------------------------------------------------------------------------------------------------------------------------------------%

%-------------------------------------------------------------------------------------------------------------------------------------------------------------------%

%TC:endignore

% Word count
%\verbatiminput{\jobname.wordcount.tex}

\end{document}